\documentclass[11pt,a4paper,twoside]{amsart}
\usepackage{amssymb,amsmath,amsthm,graphicx,color
}
\theoremstyle{plain}
\newtheorem{theorem}{Theorem}[section]

\newtheorem{lemma}[theorem]{Lemma}

\newtheorem{proposition}[theorem]{Proposition}
\newtheorem{assumption}[theorem]{Assumption}
\theoremstyle{remark}
\newtheorem{remark}[theorem]{Remark}

\allowdisplaybreaks[4]

\usepackage{amsrefs}

\numberwithin{equation}{section}

\newcommand{\C}{\mathbb{C}}
\newcommand{\R}{\mathbb{R}}
\newcommand{\Z}{\mathbb{Z}}

\newcommand{\F}{\mathcal{F}}

\renewcommand{\Im}{\operatorname{Im}}
\renewcommand{\Re}{\operatorname{Re}}
\newcommand{\I}{\infty}
\newcommand{\abs}[1]{\left\lvert #1\right\rvert}
\newcommand{\wha}[1]{\widehat{#1}}

\newcommand{\norm}[1]{\left\lVert #1\right\rVert}
\newcommand{\Lebn}[2]{\left\lVert #1 \right\rVert_{L^{#2}}}
\newcommand{\Sobn}[2]{\left\lVert #1 \right\rVert_{H^{#2}}}

\newcommand{\Jbr}[1]{\left\langle #1 \right\rangle}

\newcommand{\m}[1]{\mathcal{#1}}

\def\Sch{{\mathcal S}} 
\def\({\left(}
\def\){\right)}
\def\<{\left\langle}
\def\>{\right\rangle}
\def\le{\leqslant}
\def\ge{\geqslant}

\def\d{{\partial}}

\def \f{\phi}
\def \e{\varepsilon}

\def \d{\delta}
\def \D{\Delta}

\def \pa{\partial}
\def \n{\nabla}
\def \s{\sigma}
\def \a{\alpha}
\def \b{\beta}

\def \n{\nabla}
\def \t{\theta}

\def \P{\Phi}

\def \et{\eta}

\def \F{\mathcal{F}}
\def \mL{\mathcal{L}}
\def \mN{\mathcal{N}}
\def \mG{\mathcal{G}}
\def \mR{\mathcal{R}}

\newcommand{\g}{\gamma}

\newcommand{\eps}{\varepsilon}

\newcommand{\todayd}{\the\year/\the\month/\the\day}

\theoremstyle{definition}

\begin{document}
\title[Long range scattering for NLS]
{Long range scattering for nonlinear Schr\"odinger equations with critical homogeneous nonlinearity}
\author[S. Masaki]{Satoshi MASAKI}
\address{Division of Mathematical Science, Department of Systems Innovation, Graduate School of Engineering Science, Osaka University, Toyonaka, Osaka, 560-8531, Japan}
\email{masaki@sigmath.es.osaka-u.ac.jp}
\author[H. Miyazaki]{Hayato MIYAZAKI}
\address{Advanced Science Course, Department of Integrated Science and Technology, National Institute of Technology, Tsuyama College, Tsuyama, Okayama, 708-8509, Japan}
\email{miyazaki@tsuyama.kosen-ac.jp}
\maketitle
\vskip-5mm
\begin{abstract}

In this paper, we consider the final state problem for the nonlinear Schr\"odinger equation with
a homogeneous nonlinearity which is of the long range critical order and
is not necessarily a polynomial, in one and two space dimensions. 
As the nonlinearity is the critical order, the possible asymptotic behavior depends on the shape of the nonlinearity.
The aim here is to give a sufficient condition on the nonlinearity to construct a modified wave operator.
To deal with a non-polynomial nonlinearity,
we decompose it into a resonant part
and a non-resonant part via the Fourier series expansion.
Our sufficient condition is then given in terms of the Fourier coefficients.
In particular, we need to pay attention to the decay of the Fourier coefficients
since the non-resonant part is an infinite sum in general.

\end{abstract}
\section{Introduction}
This paper is devoted to the study of long time behavior of solutions to
the nonlinear Schr\"odinger equation
\begin{equation}\label{eq:NLS}\tag{NLS}
	i \pa_t u + \Delta u = F(u). 
\end{equation}
Here, $(t,x) \in \R^{1+d}$ ($d=1,2$) and
$u=u(t,x)$ is a $\C$-valued unknown function. 
We suppose that the nonlinearity $F$ is homogeneous of degree $1+2/d$, that is, $F$ satisfies
\begin{equation}\label{eq:cond1}
	F(\lambda u) = \lambda^{1+\frac2d} F(u)
\end{equation}
for any $\lambda >0$ and $u\in \C$. 
Our aim here is to determine the asymptotic behavior of nontrivial small solutions to \eqref{eq:NLS}
with a general homogeneous nonlinearity.
More specifically, we give a sufficient condition on the nonlinearity $F$ to construct a modified wave operator.

A typical example of nonlinearity satisfying \eqref{eq:cond1} is a gauge-invariant power type nonlinearity
\begin{equation}\label{eq:cpower}
	F(u) = \mu |u|^{\frac2d} u,
\end{equation}
where $\mu \in \R \setminus \{0\}$.
As for the nonlinearity of the form $\mu |u|^pu$,
the power $p=2/d$ is known to be a threshold.
The equation \eqref{eq:NLS} with the nonlinearity $|u|^p u$ admits a nontrivial solution asymptotically behaves
like a free solution for large time when $p>2/d$.
However, in the case $p=2/d$, there is no nontrivial solution to the equation \eqref{eq:NLS} with \eqref{eq:cpower} 
belongs to $L^2$ and scatters in $L^2$ (see \cite{B,St,TY}).
It is shown in \cite{Oz,GO} that when the nonlinearity is of the form \eqref{eq:cpower} then
for given \emph{final data} $u_+$ the equation admits a solution which asymptotically behaves like
\begin{equation}\label{eq:masymptotic}
	u_p (t) = (2it)^{-\frac{d}{2}}  e^{i \frac{|x|^2}{4t}} \wha{u_{+}}\(\frac{x}{2t}\) \exp \( -i \mu \left| \wha{u_{+}}\(\frac{x}{2t}\) \right|^{\frac{2}{d}} \log t \)
\end{equation}
as $t\to\I$, where $\wha{u_+}$ is the Fourier transform of $u_+$.

When the nonlinearity is homogeneous of the critical order,
asymptotic behavior of a solution depends on the shape of the nonlinearity.
In \cite{MTT}, it is shown that if the nonlinearity is $F(u)=u^{2/d+1}$ then
\eqref{eq:NLS} admits an asymptotic free solution
which is
a solution behaves like
\begin{equation}\label{eq:fasymptotic}
	u_p= (2it)^{-\frac{d}{2}}  e^{i \frac{|x|^2}{4t}} \wha{u_{+}}\(\frac{x}{2t}\) .
\end{equation}
Remark that this is nothing but the asymptotic behavior of the free solution $e^{it\Delta} u_+$.
Hence, the behavior in this case is similar to the case $|u|^{p}u$ ($p>2/d$).
Let us now introduce the following terminology:
We say a nonlinearity is \emph{short range} if \eqref{eq:NLS} admits a nontrivial solution
behaves like \eqref{eq:fasymptotic},
and is \emph{long range} if \eqref{eq:NLS} admits a nontrivial solution
behaves like \eqref{eq:masymptotic} with a suitable $\mu\in \R \setminus \{0\}$.
It is shown in \cite{ShT,HNST,HWN} that the nonlinearity $\mu |u|^{2/d}u +N_d (u)$ is short range if $\mu=0$
and long range if $\mu\neq0$, where
\[
	N_1(u) = \lambda_1 u^3 + \lambda_2 |u|^2 \overline{u} + \lambda_3 \overline{u}^3 
\]
if $d=1$ and
\[
	N_2(u) = \lambda_1 u^2 + \lambda_2 \overline{u}^2
\]
if $d=2$, $\lambda_j \in \C$, and $\mu \in \R \setminus \{0\}$.
Furthermore, if $\mu\neq0$ then the asymptotic behavior of a solution is given by \eqref{eq:masymptotic}.
Thus, the gauge-invariant term determines the asymptotic behavior.

In this paper, we handle more general nonlinearity satisfying \eqref{eq:cond1} and
give a sufficient condition on nonlinearity to become short range or long range.
A special example in our mind is 
\begin{equation}\label{eq:example1}
	F(u)=|\Re u|\Re u,
\end{equation}
which satisfies \eqref{eq:cond1} for $d=2$.
The nonlinearity appears, for instance, as a main part of a generalized version of Gross-Pitaevskii equation introduced in \cite{MM}.
We restrict our attention to a solution corresponding to a given final data which has very small low-frequency part.
We remark that if a final data has non-negligible low-frequency part then other kinds of asymptotic behavior take place
(see \cite{HN02,HN04,HN08,HN11,NS,HN15,N}).

With the example \eqref{eq:example1}, let us explain the main point of our argument to treat general homogeneous nonlinearity.
To compare with, let us first consider the nonlinearity $F(u)=|\Re u|^2 \Re u$ in $d=1$.
As for the nonlinearity, a simple calculation shows
\[
	|\Re u|^2 \Re u = \(\frac{u+\overline{u}}2\)^3= \frac38 |u|^2 u + \frac18 u^3 + \frac38 |u|^2 \overline{u} 
	+\frac18 \overline{u}^3
\]
Hence, this is of the form $F(u) = (3/8) |u|^2 u + N_1(u)$ and so it is included in
the previous results \cite{ShT,HNST,HWN}.
One sees that the \eqref{eq:NLS} admits a solution asymptotically behaves like \eqref{eq:masymptotic} with $\mu=3/8$.
The term $\frac38 |u|^2 u$ is a resonant term which determines the asymptotic behavior, and
the other terms are non-resonant terms.
We would emphasize that the non-resonant part is a \emph{finite sum}.
On the contrary, the resonant term of \eqref{eq:example1} is not extracted by such a simple calculation.
Hence, we use the Fourier series expansion to obtain
\[
	|\Re u| \Re u = \frac{4}{3\pi} |u| u + \sum_{m\neq 0 } 
	\frac{4(-1)^{m+1}}{\pi (2m-1)(2m+1)(2m+3)}  |u|^{-2m+1}u^{2m+1}.
\]
A remarkable point is that the non-resonant part consists of \emph{infinitely many terms}.
The question now arises whether decay of Fourier coefficients in $n$ is enough to sum up.
One main respect of the present paper is to establish a sufficient condition to handle the non-resonant part.
The condition is given in terms of the Fourier coefficients of the nonlinearity.
It will turn out that $|\Re u|\Re u$ is long range and \eqref{eq:NLS} admits
a solution which has the asymptotic \eqref{eq:masymptotic} with $\mu=4/3\pi$.

To state our result precisely, we introduce some notation.
A homogeneous nonlinearity is written as
\begin{equation}\label{eq:id1}
	F(u) = |u|^{1+\frac2d} F\( \frac{u}{|u|} \).
\end{equation}
We introduce a $2\pi$-periodic function $g(\theta)$ by
\begin{equation}\label{eq:id2}
	g(\theta) = F(e^{i\theta})
\end{equation}
We identify a homogeneous nonlinearity $F$ satisfying \eqref{eq:cond1}
with a $2\pi$-periodic function $g$ through \eqref{eq:id1} and \eqref{eq:id2}.
Namely, given nonlinearity $F$, we give a $2\pi$-periodic function $g(\theta)=g_F(\theta)$ by the above procedure.
Conversely, for a given $2\pi$-periodic function $g$, we can construct a homogeneous nonlinearity $F=F_g:\C \to \C$ by
\[
	F_g(u) =
	\left\{
	\begin{aligned}
	&|u|^{1+\frac2d} g\( \arg u \), &&u\neq 0, \\
	&0 &&u=0.
	\end{aligned}
	\right.
\]
We now apply the Fourier series expansion.
Since $g(\theta)$ is $2\pi$-periodic function, it holds, at least formally, that
$g (\theta) = \sum_{n\in\Z} g_n e^{in\theta}$ with the coefficients
\begin{equation}\label{def:gn}
	 g_n := \frac{1}{2\pi} \int_0^{2\pi} g(\t) e^{-in\t}d\t.
\end{equation}
This expansion gives us
\[
	F(u) = |u|^{\frac2d+1}\sum_{n\in \Z} g_n \(\frac{u}{|u|}\)^n = \sum_{n\in \Z} g_n |u|^{1+\frac2d-n} u^n,
\]
by means of \eqref{eq:id1} and \eqref{eq:id2}.
We then write
\begin{equation}\label{eq:decomp}
	F(u) = g_0 |u|^{\frac2d+1} + g_1 |u|^{\frac2d}u + \sum_{n\neq 0,1} g_n |u|^{1+\frac2d-n} u^n.
\end{equation}
The extraction of a resonant term via Fourier expansion is motivated by \cite{Sun,MS2}.
We also remark that the Fourier coefficients are represented as the integral
$g_n=\frac1{2\pi i} \int_{|z|=1} F(z) \frac{dz}{z^{1+n}}$, some of which are used in previous works such as 
\cite{MR3543568}.

In this paper, we suppose the following.
\begin{assumption}\label{asmp:main}
Assume that $F$ is a homogeneous nonlinearity such that a corresponding $2\pi$-periodic function
$g(\theta)$ satisfies
\[
	g_0 := \frac1{2\pi}\int_{0}^{2\pi} g(\theta) d\theta  =0, \quad g_1 := \frac1{2\pi} \int_{0}^{2\pi} g(\theta)e^{-i\theta} d\theta \in \R,
\]
and $\sum_{n \in \Z}|n|^{1+\eta} |g_n| < \I$
for some $\eta>0$, where $g_n$ is given in \eqref{def:gn}.
In particular, $g$ is Lipschitz continuous.
\end{assumption}

\subsection{Main result}
Set $\Jbr{a}=(1+|a|^2)^{1/2}$ for any $a \in \C$. The weighted Sobolev space on $\R^d$ is defined by $H^{m,s}=\{ u \in \Sch'(\R^d); \; \Jbr{i\n}^{m}\Jbr{x}^s u \in L^2 \}$, and $\dot{H}^{m}=\{ u \in \Sch'(\R^d); \; (-\D)^{\frac{m}2}u  \in L^2 \}$ denotes the homogeneous Sobolev space on $\R^d$, where $m$, $s \in \R$. Let us simply write $H^{m} = H^{m,0}$.
Let $\norm{g}_{\mathrm{Lip}} := \sup_{\t \neq \t'} {|g(\t)-g(\t')|}/{|\t-\t'|}$ be the Lipschitz norm.

Our main result is as follows:

\begin{theorem}\label{thm:main}
Suppose that the nonlinearity $F$ satisfies Assumption \ref{asmp:main} for some $\eta>0$.
Fix $\d \in (d/2,(d+1)/2)$ so that $\delta-d/2<2\eta$.
Let $\gamma=\d/2$ if $d=1$ and $\gamma = (\d+2)/6$ if $d=2$. 
Take $ b \in(d/4, \gamma)$.
Then, there exists $\eps_0=\eps_0(b,\norm{g}_{\mathrm{Lip}})$ such that 
for any $u_+ \in H^{0,d} \cap \dot{H}^{-\d}$ satisfying
 $\Lebn{\wha{u_+}}{\I}<\eps_0$, there exists $T>0$ and a unique solution $u \in C([T,\I); L^2(\R^d))$ of \eqref{eq:NLS} such that 
\[
\sup_{t \in [T, \I)}t^{b} \Lebn{u(t) -u_p(t)}{2} + \sup_{t \in [T, \I)} t^b \( \int_t^{\I} \norm{u(s)-u_p(s)}_{X_d}^4 ds \)^{\frac{1}{4}} < \I,
\]
where
\[
u_p (t) = (2it)^{-\frac{d}{2}}  e^{i \frac{|x|^2}{4t}} \wha{u_{+}}\(\frac{x}{2t}\) \exp \( -i g_1 \left| \wha{u_{+}}\(\frac{x}{2t}\) \right|^{\frac{2}{d}} \log t \),
\]
$X_1 = L^{\I}(\R)$ and $X_2 = L^{4}(\R^2)$.
\end{theorem}

\begin{remark}
Our theorem include the example \eqref{eq:example1} in $d=2$.
The corresponding periodic function is
$g(\theta) = |\cos \theta| \cos \theta$
and so 
\[
g_n = \left\{ 
\begin{aligned}
& - \frac{4}{\pi (n-2)n(n+2)} \sin \frac{\pi n}{2}  && n\text{: odd},\\
& 0 && n \text{: even}.
\end{aligned}
\right.
\]
Remark that it satisfies Assumption \ref{asmp:main} for $0<\eta<1$.
\end{remark}

\begin{remark}
The regularity assumption on the data is similar to that is made in \cite{HNST} and stronger than in \cite{HWN}.
This is because we use regularity of the data to weaken the condition on the nonlinearity.
Hence, if $F$ is a sufficiently good one, for instance if it satisfies Assumption \ref{asmp:main} with $\eta\ge d$,
then the regularity assumption can be taken the same as in \cite{HWN} by their argument.
\end{remark}

\begin{remark}
Theorem \ref{thm:main} implies that \eqref{eq:NLS} admits a nontrivial asymptotic free solution when $F$ satisfies Assumption \ref{asmp:main} and $g_1=0$.
For example, $F(u)=|\Re u | \Re u - i |\Im u| \Im u$, $d=2$, is short range.
Indeed, the corresponding periodic function is $g(\t)=|\cos \t|\cos \t - i |\sin \t|\sin \t$ and so
\[
	g_n = \left\{
	\begin{aligned}
	& \frac{8}{\pi(n-2)n(n+2)}, &&n\equiv3 \mod 4, \\
	& 0, && \text{otherwise}.
	\end{aligned}
	\right.
\]
\end{remark}

The rest of the paper is organized as follows.
In the next section, we outline the proof.
Then, it will turn out that the main step of the proof is the estimate of non-resonant part (Proposition \ref{non:im1}).
After summarizing several useful estimates in Section \ref{sec:preliminary},
we prove Proposition \ref{non:im1} in Section \ref{sec:mainlem}.
Main theorem is then shown in Section \ref{sec:main_pf}.

\section{Outline of the proof}
By the decomposition \eqref{eq:decomp} and Assumption \ref{asmp:main}, we write
\[
	F(u) = g_1 |u|^{\frac2d}u + \sum_{n\neq 0,1} g_n |u|^{\frac2d-n} u^n.
\]
Denote
\[
	\mG_d (u) = g_1 |u|^{\frac2d}u, \quad  \mN_d(u)  = \sum_{n\neq 0,1} g_n |u|^{1+\frac2d-n} u^n.
\]
The heart of matter is that the expansion \eqref{eq:decomp} successfully extracts a ``resonant part'' $\mG_d$
which determines the shape of asymptotic behavior $u_p$.
The validity of the extraction is confirmed by proving the other part, a ``non-resonant part'' $\mN_d$, enjoys better time decay.
The decay comes from the fact that the phase of the non-resonant part is different from that of linear part.
In the integral form of the equation,
it can be seen that this disagreement of phase actually causes a time decay effect (cf.~ stationary phase).
This kind of additional decay was known for the case where
$\mN_d$ is a specific \emph{finite sum} of $|u|^{1+2/d-n}u^n$ ($n\neq0,1$) (see \cite{HNST,HWN}).
However, our non-resonant part is \emph{an infinite sum}.
In the technical point of view, a contribution of this paper is
a treatment of the infinite sum under Assumption \ref{asmp:main}.

We introduce a formulation in \cite{HWN} (see also \cite{HNST,ShT}).
In what follows, we let $t>1$ unless otherwise stated.
Let $U(t)=e^{it\D}$.
Introduce a multiplication operator $M(t)$ and a dilation operator $D(t)$ by
\begin{equation}\label{def:MD}
	M(t) = e^{\frac{i |x|^2}{4t}}, \quad (D(t)f)(x) = (2it)^{-\frac{d}{2}} f\(\frac{x}{2t}\).
\end{equation}
They are isometries on $L^2(\R^d)$. Then,
\begin{equation}\label{def:upw}
	u_p(t) = M(t)D(t) \widehat{w}(t), \quad \widehat{w}(t) = \wha{u_{+}} \exp(-i g_1 |\wha{u_{+}}|^{\frac{2}{d}} \log t).
\end{equation}
We regard the equation \eqref{eq:NLS} as
\[
\mL(u-u_{p}) = (F (u) -F (u_p)) - \mL u_p + \mG_{d} (u_{p}) + \mN_d(u_p),
\]
where $\mL=i\pa_t + \Delta_x$.
A computation shows that it is rewritten as the following integral equation;
\begin{equation}
\begin{aligned}
	u(t) -u_p(t) 
	&= i \int_t^{\I} U(t-s) \( F (u) - F (u_p)  \)(s) ds \\ 
	& +\mR(t) \wha{w} -i \int_t^{\I} U(t-s) \mR (s) \mG_d (\wha{w}) (s) \frac{ds}{s} \\
	& +i \int_t^{\I} U(t-s)\mN_d (u_p)(s) ds,
\end{aligned}
\label{inteq:1}
\end{equation}
where
\[
\mR(t) = M(t) D(t)\(U\(-\frac{1}{4t}\) -1\)
\]
(see \cite{HWN}).

Let $X_1=L^\I(\R)$ and $X_2=L^4(\R^2)$.
For $R,T,b>0$, we define a complete metric space
\begin{align*}
X_{d, T, b, R} &:= \{v \in C([T,\I); L^2(\R^d)) ;\  \norm{v-u_p}_{X_{d, T, b}} \le R \}, \\
\norm{v}_{X_{d,T,b}} &:= \sup_{t \in [T, \I)}t^{b} \norm{v(t)}_{L^2(\R^d)} + \sup_{t \in [T, \I)} t^b \( \int_t^{\I} \norm{v(s)}_{X_d}^4 ds \)^{\frac{1}{4}}, \\
d (u,v) &:= \norm{u-v}_{X_{d,T,b}} .
\end{align*}
We shall show that, under the assumption of Theorem \ref{thm:main},
 there exists $\eps_0>0$ such that for any data $u_+ \in H^{0,d} \cap \dot{H}^{-\d}$ with
$\Lebn{\wha{u_+}}{\I} \le \e_0$, we can choose $R,T>0$ so that
the map
\begin{equation}\label{def:Phi}
\begin{aligned}
\P(v)(t) :={}& u_p(t) + i \int_t^{\I} U(t-s) \( F (v) - F (u_p)  \)(s) ds  \\ 
& +\mR(t) \wha{w} -i \int_t^{\I} U(t-s) \mR (s) \mG_d (\wha{w}) (s) \frac{ds}{s}  \\
& +i \int_t^{\I} U(t-s)\mN_d (u_p)(s) ds. 
\end{aligned}
\end{equation}
is a contraction map on $X_{d, T, b, R}$.

To this end, we introduce three intermediate results.
The first one is a consequence of Strichartz' estimate.
\begin{lemma} \label{non1:g2}
Let $\wha{u_+} \in L^\infty$. Assume that $g(\theta)$ is Lipschitz continuous. 
If $b>d/4$ then it holds that
\begin{multline*}
\norm{\int_t^{\I}U(t-s) \( F (v) - F (u_p)\) ds}_{X_{d,T,b}} \\
\le C \norm{v-u_p}_{X_{d,T,b}} \( \norm{v-u_p}_{X_{d,T,b}}^{\frac{2}{d}}T^{\frac{1}{2} -\frac{2}{d}b} + \Lebn{\wha{u_+}}{\I}^{\frac{2}{d}} \), 
\end{multline*}
where $C$ depends on the Lipschitz constant of $g$.
\end{lemma}
The estimate is essentially the same as in \cite{HWN,HNST,ShT}.
Remark that Lipschitz continuity of $g$ gives us
\begin{align*}
	\left| F(v) - F (u_p) \right| 
	\le C  \( \left| v-u_p \right|^{1+ \frac{2}{d}} + |u_p|^{\frac{2}{d}} \left|v - u_p\right| \).
\end{align*}
The detail is given in Appendix \ref{sec:A}.

The main step is the estimate of ``external terms'' on the right hand side of \eqref{def:Phi}.
The second one is due to Hayashi, Wang, and Naumukin \cite[Lemma 2.1]{HWN}.
\begin{lemma}\label{non1:g1}
Let $u_+ \in H^{0,d}$ and $d/2<\d<d$. Then, the estimates
\begin{multline*}
\norm{\mR (t) \wha{w}}_{L^{\I}_t (T, \I ; L^2)} + \norm{\mR (t) \wha{w}}_{L^{4}_t (T, \I ; X_d)} \\
\le CT^{-\frac{\d}{2}}\Jbr{g_1\Lebn{\wha{u_+}}{\I}^{\frac{2}{d}} \log T}^\d \norm{u_+}_{H^{0,d}},
\end{multline*}
and
\begin{align*}
&\norm{\int_t^{\I}U(t-s)\mR(s) \mG_d (\wha{w}) \frac{ds}{s}}_{L^{\I}_t (T, \I ; L^2)} + \norm{\int_t^{\I}U(t-s)\mR(s) \mG_d (\wha{w}) \frac{ds}{s}}_{L^{4}_t (T, \I ; X_d)} \\
&\le C |g_1| T^{-\frac{\d}{2}} \Jbr{g_1\Lebn{\wha{u_+}}{\I}^{\frac{2}{d}} \log T}^\d \Lebn{\wha{u_+}}{\I}^{\frac{2}{d}} \norm{u_+}_{H^{0,d}} 
\end{align*}
hold for all $T>1$.
\end{lemma}

The last one, in which the main technical issue lies, 
is an estimate on the term $\mN_d (u_p)$.
\begin{proposition} \label{non:im1}
Let $u_+ \in H^{0,d} \cap \dot{H}^{-\d} $ with $d/2 < \d < (d+1)/2$. Assume that $\sum_{n\in\Z}|n|^{1+\eta}|g_n| < \I$
for some $\eta>\frac12(\delta-\frac{d}2)$.
Then, the estimate
\begin{multline*}
	\norm{\int_t^{\I}U(t-s) \mN_d (u_p) ds}_{L^{\I}_t (T, \I ; L^2)} + \norm{\int_t^{\I}U(t-s) \mN_d (u_p) ds}_{L^{4}_t (T, \I ; X_d)} \\
	\le CT^{-\g} \Jbr{ g_1 \norm{\wha{u_+}}_{L^\I}^{\frac2{d}} \log T}^\d \Jbr{g_1\Lebn{\wha{u_+}}{\I}^{\frac{2}{d}}}\\
	\times   \Lebn{\wha{u_+}}{\I}^{\frac{2}{d}} \norm{u_+}_{H^{0, d} \cap \dot{H}^{-\d}} \sum_{n \ne 0,1} |n|^{1+\eta} |g_n|
\end{multline*}
holds for all $T>1$, where $\norm{u_+}_{H^{0,d} \cap \dot{H}^{-\d}} = \norm{u_+}_{H^{0,d}} + \norm{u_+}_{\dot{H}^{-\d}}$.
\end{proposition}

As for the assumption on the nonlinearity,
the assumption of Proposition \ref{non:im1} is stronger than that of Proposition \ref{non1:g2}
because if $g$ satisfies $\sum_{n} |n|^{1+\eta} |g_n| <\infty$ then it is Lipschitz continuous.
The assumption of the Theorem \ref{thm:main} comes from this proposition
in order to estimate Sobolev norm of the nonlinearity.
Roughly speaking, $s$ time derivative of $|u|^{1+2/d-n}u^n$ produces $O(|n|^s)$.
Hence, to weaken the assumption of the nonlinearity we shall use as less derivative as possible.
We remark again that we have to pay attention to the above growth order just because we are 
working with the non-resonant term which consists of infinitely many terms.
Our proof is in the same spirit as in \cite{HWN}.
However, the argument in \cite{HWN} works only for large $\eta$.
We introduce two techniques to handle small $\eta$.
Especially, they are necessary to include the example \eqref{eq:example1}.
The detail of the technique is discussed in the forthcoming section.

\section{Key estimates}\label{sec:preliminary}

We introduce two techniques to weaken the assumption on the nonlinearity.
The argument in \cite{HWN} works only for large $\eta$.

\subsection{Estimates on nonlinearity}
The first one is related to estimation of 
 $\||\wha{w}|^{1+\frac2d-n}\wha{w}^n\|_{H^\d}$.
One easily obtains such an estimate via an equivalent difference characterization of the norm
of the corresponding Besov space $B^\d_{2,2}$.
However, a straightforward calculation in this direction gives us no more than an upper bound of order $O(n^d)$
(remark that here $d$ equals the minimum integer larger than $\d$).
Hence, we use an interpolation inequality to improve the order into $O(n^\delta)$
in exchange for strengthening the regularity assumption on the data. 
This is the first technique.

Let us begin with a preliminary estimate.
\begin{lemma} \label{lem1:1}
For $n \in \Z$, it holds that
\begin{align*}
\norm{|u|^{1+\frac2d - n}u^{n}}_{{H}^{d}} &\le C\Jbr{n}^{d} \norm{u}_{L^\I}^{\frac2d}\norm{u}_{{H}^{d}}
\end{align*}
for any $u \in {H}^{d}(\R^d)$. 
\end{lemma}
The lemma is obvious by $\norm{f}_{H^d}^2\sim \sum_{\alpha \in (\Z_{\ge0})^d,\,|\alpha|\le d}\norm{\pa_x^\alpha f}_{L^2}^2$.
Then, we have the following estimate.
\begin{lemma} \label{lem1:2}
Let $\wha{w}$ be as in \eqref{def:upw}.
Then, it holds that
\begin{align*}
\norm{\wha{w}}_{H^{d}} &\le C\norm{u_+}_{H^{0, d}}\Jbr{ g_1 \norm{\wha{u_+}}_{L^\I}^{\frac2{d}} \log t}^d,\\
\norm{\partial_t\wha{w}}_{H^{d}} &\le C\frac{|g_1|}{t}\norm{\wha{u_+}}_{L^\I}^{\frac2{d}} \norm{u_+}_{H^{0, d}}\Jbr{ g_1 \norm{\wha{u_+}}_{L^\I}^{\frac2{d}} \log t}^d.
\end{align*}
Moreover, 
\begin{align*}
	\norm{|\wha{w}|^{1+\frac2d - n}{\wha{w}}^{n}}_{{H}^{\d}}
	&\le C \Jbr{n}^\d
	 \norm{\wha{u_+}}_{L^\I}^{\frac2d} 
	\norm{u_+}_{H^{0, d}}\Jbr{ g_1\norm{\wha{u_+}}_{L^\I}^{\frac2d} \log t}^\d, \\
	\norm{\partial_t (|\wha{w}|^{1+\frac2d - n}{\wha{w}}^{n})}_{{H}^{\d}}
	&\le C\frac{\Jbr{n}^{1+\d}|g_1|}{t}\norm{\wha{u_+}}_{L^\I}^{\frac4{d}} \norm{u_+}_{H^{0, d}}\Jbr{ g_1 \norm{\wha{u_+}}_{L^\I}^{\frac2{d}} \log t}^\d
\end{align*}
for any $0 \le  \d \le d$ and $t\ge1$.
\end{lemma}
\begin{proof}
The first estimate is immediate.
By interpolation inequality, H\"older's inequality, and Lemma \ref{lem1:1}, we have
\begin{align*}
	\norm{|\wha{w}|^{1+\frac2d - n}{\wha{w}}^{n}}_{{H}^{\d}}
	&\le \norm{|\wha{w}|^{1+\frac2d - n}{\wha{w}}^{n}}_{L^2}^{1-\frac{\delta}d}
	\norm{|\wha{w}|^{1+\frac2d - n}{\wha{w}}^{n}}_{{H}^{d}}^{\frac{\delta}d}\\
	&\le C\Jbr{n}^\delta \norm{\wha{w}}_{L^\I}^{\frac2d} \norm{\wha{w}}_{L^2}^{1-\frac{\d}d} \norm{\wha{w}}_{H^d}^{\frac{\d}d}.
\end{align*}
Then, the third estimate is a consequence of the first.

Let us next prove the second. We only consider $d=2$. The other case is similar.
By definition, we have
\[
	\partial_t \wha{w} = - \frac{ig_1}{t}
	|\wha{u_{+}}| \wha{u_{+}} \exp(-i g_1 |\wha{u_{+}}| \log t).
\]
Hence, by the Schwarz inequality, one sees that
\begin{align*}
	|\nabla^2 \partial_t \wha{w}| \le C\frac{|g_1|}{t} (|\wha{u_{+}}||\nabla^2 \wha{u_+}|+
	|\nabla\wha{u_{+}}|^2)
	+ C\frac{|g_1|^3 (\log t)^2}{t} 
	|\wha{u_{+}}||\nabla\wha{u_{+}}|^2.
\end{align*}
Then, a use of Gagliardo-Nirenberg inequality yields
\[
	\norm{\partial_t \wha{w}}_{\dot{H}^2}
	\le C \frac{|g_1|}{t} \norm{\wha{u_+}}_{L^\I}\norm{\wha{u_+}}_{\dot{H}^2}
	+ C\frac{|g_1|^3 (\log t)^2}{t} 
	\norm{\wha{u_+}}_{L^\I}^{3}\norm{\wha{u_+}}_{\dot{H}^2}.
\]
Plugging this to the trivial $L^2$ estimate, we obtain the second estimate.

To prove the last estimate,
we remark that $\partial_t (|\wha{w}|^{1+\frac2d-n} \wha{w}^n)$ is of the form
$$\frac{g_1}{t}(P_1(\wha{u_{+}}) \exp(-i g_1 |\wha{u_{+}}| \log t)
+P_2(\wha{u_{+}})\exp(i g_1 |\wha{u_{+}}| \log t))$$
with polynomials $P_j(z)= O(\Jbr{n}|z|^{\frac{4}d-n}z^{n})$ and so that we 
can obtain
\[
	\norm{\partial_t (|\wha{w}|^{1+\frac2d-n} \wha{w}^n)}_{\dot{H}^d} \le
	C\frac{\Jbr{n}^{1+d}|g_1|}{t}\norm{\wha{u_+}}_{L^\I}^{\frac4{d}} \norm{u_+}_{H^{0, d}}\Jbr{ g_1 \norm{\wha{u_+}}_{L^\I}^{\frac2{d}} \log t}^d
\]
as in the second estimate.
Then, mimicking the proof of the third estimate, we obtain the desired estimate.
\end{proof}

\subsection{Time-dependent regularizing operator}
To obtain additional time decay property of non-resonant part $\mN_d(u_p)$,
we use integration by parts in time.
However, the argument requires spatial regularity.
In \cite{HWN}, Hayashi, Wang, and Naumkin introduce
a \emph{time-dependent} regularizing operator (or a \emph{time-dependent} cutoff to low-frequency),
and reduce required regularity by applying the above integration by parts only for a low-frequency part
and by estimating the remaining high-frequency part with the fact that the operator converges to the identity operator as $t\to\I$.

In this paper, we take this kind of regularizing operator \emph{dependently on both $t$ and $n$}.
This is the second technique to weaken the assumption on the nonlinearity.

Let $\psi \in \mathcal{S}$.
We introduce a regularizing operator $\m{K}_\psi=\m{K}_\psi(t,n)$ by
\begin{equation}\label{def:Kpsi}
	\m{K}_\psi := \psi\( \frac{i\nabla}{|n| t^{\s/2}} \):=\F^{-1} 
	\psi \( \frac{\xi}{|n| t^{\s/2}} \) \F,
\end{equation}
where $\sigma =1$ if $d=1$ and $\sigma = \frac{2+\delta}{3}>1$ if $d=2$.
We have
\[
\m{K}_\psi f = C_d ((|n| t^{\s/2} )^{d}  \F^{-1}\psi( |n| t^{\s/2} \cdot ) * f)(x).
\]
\begin{lemma}[Boundedness of $\m{K}$] \label{mol:1.1}
Take $\psi \in \mathcal{S}$ and set $\m{K}_\psi$ as in \eqref{def:Kpsi}.
Let $s\in \R $ and $\theta \in [0,1]$. For any $t>0$ and $n\neq 0$, the followings hold.
\begin{enumerate}
\renewcommand{\labelenumi}{(\roman{enumi})}
\item
$\m{K}_\psi$ is a bounded linear operator on $L^2$ and satisfies $\norm{\m{K}_\psi}_{\m{L}(L^2)} \le \norm{\psi}_{L^\I}$.
Further, $\m{K}_\psi$ commutes with $\nabla$. In particular, $\m{K}_\psi$ is a bounded linear operator on $\dot{H}^s$ and satisfies $\norm{\m{K}_\psi}_{\m{L}(\dot{H}^s)} \le \norm{\psi}_{L^\I}$.
\item 
$\m{K}- {\psi}(0)$ is a bounded linear operator from $\dot{H}^s$ to $\dot{H}^{s+\theta}$ with norm
$$\norm{\m{K}_\psi-\psi(0)}_{\m{L}(\dot{H}^{s+\t} ,\dot{H}^{s})} \le Ct^{-\frac{\theta \s}{2}}|n|^{-{\t}}.$$
\end{enumerate}
\end{lemma}
\begin{proof}
The first item is obvious.
Let us prove the second.
It suffices to show the case $s=0$.
For $f \in \dot{H}^{\theta}$, one sees from the equivalent expression that
\begin{align*}
\Lebn{(\m{K}_\psi-{\psi}(0))\f}{2} &\le C_d ({|n| t^{\s/2}} )^{d} \int_{\R^d} |\F^{-1} \psi(|n| t^{\s/2}\eta)| \Lebn{\f(\cdot- \eta) - \f}{2} d\eta \\
&\le C ({|n| t^{\s/2}} )^{d} \int_{\R^d} |\F^{-1} \psi(|n| t^{\s/2}\eta)| \Lebn{ \sin\frac{ \xi \cdot \eta}2  \F\f}{2} d\eta \\
&\le C ({|n| t^{\s/2}} )^{d} \int_{\R^d} |\F^{-1} \psi(|n| t^{\s/2}\eta)| |\et|^{\theta} \Lebn{|\xi|^{\theta} \F\f}{2} d\eta \\
&\le C_\psi t^{-\frac{\theta \s}{2}}|n|^{-{\t}}\norm{\f}_{\dot{H}^\theta}.
\end{align*}
The proof is completed.
\end{proof}

\section{Proof of Proposition \ref{non:im1}}\label{sec:mainlem}
In this section, we prove Proposition \ref{non:im1}. 
Using $u_p  = M(t)D(t)\wha{w}(t) = D(t)E(t)\wha{w}(t)$ with $E(t) = e^{it|x|^2/2}$, we obtain
\[
	\mN_d (u_p) = \sum_{n \ne 0,1} g_n \( \frac{1}{it} D(t) E^n(t) \phi_n (t) \),
\]
where
\begin{equation}\label{def:phin}
	\phi_n(t):= |\wha{w}(t)|^{1+\frac{2}{d}-n}\wha{w}^n(t).
\end{equation}
Let $\psi_0(x) = e^{-|x|^2/4} \in \m{S}$ and set
$\m{K}(t,n):=\m{K}_{\psi_0}(t,n)$ as in \eqref{def:Kpsi}
with $\s = 1$ if $d=1$ and $\s = \frac{2+\d}{3}>1$ if $d=2$.
We decompose $\mN_d(u_p)$ into low frequency part and high frequency part,
\[
	\mN(u_p) = \m{P}_d + \m{Q}_d, 
\]
where
\begin{align*}
	\m{P}_d &= \sum_{n \ne 0,1} g_n \( \frac{1}{it} D(t)\( E^n(t) \m{K} \phi_n(t) \) \), \\
	\m{Q}_d &= -\sum_{n \ne 0,1} g_n \( \frac{1}{it} D(t)\( E^n(t) (\m{K}-1) \phi_n(t) \) \).
\end{align*}

We estimate high frequency part $\m{Q}_d$.
By Strichartz' estimate, it suffices to bound $\norm{\m{Q}_d}_{L^1(T,\infty;L^2)}$.
By using Lemma \ref{mol:1.1} (ii) and Lemma \ref{lem1:2}, we have
\begin{align*}
	\Lebn{\m{Q}_d(t)}{2} &\le Ct^{-1} \sum_{n \ne 0,1}|g_n| \Lebn{(\m{K}-1) \phi_n}{2} \\
	&\le Ct^{-1-\frac{\t\s}{2}} \sum_{n \ne 0,1} |n|^{-\t} |g_n| \norm{\phi_n}_{\dot{H}^\theta} \\
	&\le Ct^{-1-\frac{\t\s}{2}}
	 \norm{\wha{u_+}}_{L^\I}^{\frac2d} 
	\norm{u_+}_{H^{0, d}}\Jbr{ g_1\norm{\wha{u_+}}_{L^\I}^{\frac2d} \log t}^\t
	\sum_{n \ne 0,1} |g_n|.
\end{align*}
We choose $\theta= \d < 1$ if $d=1$ and $\t = 1$ if $d=2$. Then, we obtain
\begin{multline}
	\norm{\m{Q}_d}_{L^1(T,\infty;L^2)} \\ 
	\le CT^{-\gamma}\Jbr{ g_1\norm{\wha{u_+}}_{L^\I}^{\frac2d} \log T}^\d
 \norm{\wha{u_+}}_{L^\I}^{\frac2d}\norm{u_+}_{H^{0, d}} 
	\sum_{n \ne 0,1} |g_n|. 
	\label{non:n4}
\end{multline}

Next, we estimate low frequency part $\norm{\int_t^{\I}U(t-s) \m{P}_d (u_p) ds}_{X_d}$. 
By the factorization of $U(t)$,
\[
	U(t) = M(t)D(t)\F M(t) = M(t)D(t)U\(-\frac{1}{4t}\)\F.
\]
Further, the Gagliardo-Nirenberg inequality implies $\Lebn{F}{p} \le C \Sobn{F}{\nu}^{a}\Lebn{F}{2}^{1-a}$ for $p \ge 2$ and $a \in [0,1)$ with $\frac{1}{p} = \frac{1}{2} - \frac{a \nu}{d}$. 
Hence,
\begin{equation}\label{eq:target}
\begin{aligned}
	&\norm{\int_t^{\I}U(t-s) \m{P}_d(s) ds}_{X_d} \\
	&= \norm{U(t)\F^{-1} \int_t^{\I}\F U(-s) \m{P}_d(s) ds}_{X_d} \\
	&= \norm{D(t)U\(-\frac{1}{4t}\) \int_t^{\I}\F U(-s) \m{P}_d(s) ds}_{X_d} \\
	&\le Ct^{ -\frac{1}{2}} \norm{\int_t^{\I}\F U(-s) \m{P}_d(s) ds}_{{H}^\nu}^{a} \norm{\int_t^{\I}\F U(-s) \m{P}_d(s) ds}_{L^2}^{1-a} 
\end{aligned}
\end{equation}
for $\nu=1/2a>1/2$. 
We fix $\nu$ so that
\begin{equation}\label{def:nu}
\frac12  < \nu <  \min \(\delta, 2-\delta\) .
\end{equation}
To choose such $\nu$, we need $\delta<3/2$.

By factorization of $U(t)$, we have 
\[
	\F U(-t) D(t) E^{\rho}(t) = i^{\frac{d}{2}} E^{1-\frac{1}{\rho}}(t)U\( \frac{\rho}{4t} \) D\(\frac{\rho}{2}\)
\] 
for $\rho \ne 0$ (see \cite{HWN}). Therefore, we further compute
\begin{align*}
	\F U(-s) \m{P}_d(s) &= \sum_{n \ne 0,1} g_n \F U(-s)\frac{1}{i s}D(s) E^{n}(s) \m{K}\phi_n (s) \\
	&= \sum_{n \ne 0,1} g_n s^{-1} i^{\frac{d}{2}-1} E^{1-\frac{1}{n}}(s) U\( \frac{n}{4s} \) D\(\frac{n}{2}\) \m{K} \phi_n(s).
\end{align*}
Now, we have $E^{1-\frac{1}{n}}(s) = A(s) \pa_s (s E^{1-\frac{1}{n}}(s))$ for $n\neq0,1$, where $A(s) = \( 1+\frac{i(1-\frac{1}{n})s}{2}|\xi|^2 \)^{-1}$.
Further,
\[
	\pa_{s}U \(\frac{n}{4s}\) = U\( \frac{n}{4s}\) \( \pa_s - \frac{i n}{4s^2} \D \).
\]
Therefore, an integration by parts gives us
\begin{align}
\begin{aligned}
	&\int_t^{\I} E^{1-\frac{1}{n}}(s) U\(\frac{n}{4s} \) D\(\frac{n}{2}\) \m{K}\f_n(s) \frac{ds}{s} \\
	&= -E^{1-\frac{1}{n}}(t) A(t) U\(\frac{n}{4s} \) D\(\frac{n}{2}\) \m{K}\f_n(t) \\
	&- \int_t^{\I} E^{1-\frac{1}{n}}(s) s \pa_s \(s^{-1}A(s)\) U\(\frac{n}{4s} \) D\(\frac{n}{2}\) \m{K}\f_n(s) ds \\
	&- \int_t^{\I} E^{1-\frac{1}{n}}(s) A(s) U\(\frac{n}{4s} \) \( \pa_s - \frac{i n}{4s^2} \D \) D\(\frac{n}{2}\) \m{K}\f_n(s) ds \\
	&=: I_1 +I_2 +I_3.
\end{aligned}
	\label{im1:1}
\end{align}

Thanks to \eqref{eq:target}, we shall estimate $I_j$ ($j=1,2,3$) in $L^2$ and ${H}^\nu$.
The following estimate is useful.
\begin{lemma} \label{useful:1:1}
Let $d/2 < \d < (d+1)/2$ and 
$\d<d/2+2\eta$.
Let $\nu$ satisfy either $\nu=0$ or
$1/2< \nu < \min (\delta, 2 -\delta)$. 
Let $\b = \max(1,\d)$ and let $m=1$, $2$. 
Then, it holds for any $t\ge1$ and $n\neq 0,1$ that
\begin{equation}\label{eq:useful}
\begin{aligned}
	&\Sobn{E^{1-\frac{1}{n}}(t) A^{m}(t) U\(\frac{n}{4t} \) D\(\frac{n}{2}\) \m{K}\f_n(t)}{\nu} \\
	&\quad\le C t^{\frac{\nu-\d}{2}}|n|^{-\d+\eta} \( \Sobn{\phi(t)}{\d} + \Lebn{|\xi|^{-\d}\f_n(t)}{2} \) \\
	&\quad\quad+C t^{\frac{\nu-\a}{2}} |n|^{-\a+\eta} \( \norm{\f_n(t)}_{H^{\d}} + \Lebn{|\xi|^{-\d}\f_n(t)}{2} \)^{1-\nu}  \norm{\f_n(t)}_{H^{\b}}^{\nu},
\end{aligned}
\end{equation}
where $\alpha = (1-\nu)\d + \nu \b$.
\end{lemma}
We postpone the proof of this lemma and continue the proof of Proposition \ref{non:im1}.
For simplicity, we consider the case $d=2$, in which case $\a = \b=\d$ in Lemma \ref{useful:1:1}.
Fix $\eta>\frac12 \(  \d -\frac{d}2\)$.
Using Lemma \ref{useful:1:1}, we obtain
\begin{align}
\begin{aligned}
	\Sobn{I_1}{\nu}&=\Sobn{E^{1-\frac{1}{n}}(t) A(t) U\(\frac{n}{4t} \) D\(\frac{n}{2}\) \m{K}\f_n(t)}{\nu} \\
	&\le C t^{\frac{\nu-\d}{2}}|n|^{-\d+\eta} \norm{\phi(t)}_{H^{\d} \cap H^{0, -\d}}.
\end{aligned}
\label{main:1:1}
\end{align}
Let us estimate $\Sobn{I_2}{\nu}$. By $\pa_s \(s^{-1}A(s)\) = -2s^{-2} A(s) + s^{-2} \(A(s)\)^2$ and Lemma \ref{useful:1:1}, we compute
\begin{align}
\begin{aligned}
	\Sobn{I_2}{\nu} &\le C \int_t^{\I} \Sobn{E^{1-\frac{1}{n}}(s) s^{-1}A(s) U\(\frac{n}{4s} \) D\(\frac{n}{2}\) \m{K}\f_n(s) }{\nu} ds \\
	&\quad + C\int_t^{\I} \Sobn{E^{1-\frac{1}{n}}(s) s^{-1}\(A(s)\)^{2} U\(\frac{n}{4s} \) D\(\frac{n}{2}\) \m{K}\f_n(s)}{\nu} ds \\
	&\le C|n|^{-\d+\eta} \int_t^{\I} s^{\frac{\nu-\d}{2}-1} \norm{\phi(s)}_{H^{\d} \cap H^{0, -\d}} ds.
\end{aligned}
\label{main:1:2}
\end{align}

Finally, we estimate $\Sobn{I_3}{\nu}$. We introduce the regularizing operators
$\m{K}_j := \m{K}_{\psi_j}$ ($j=1,2$) by \eqref{def:Kpsi} with
\begin{align*}
	&\psi_{1}(x) = -\frac{\s}2 x\cdot \nabla \psi_0 \in \Sch, &
	&\psi_{2}(x) = \frac{i}4  |x|^2 \psi_0(x) \in \Sch.
\end{align*}
We then have an identity
\begin{align*}
	\(\pa_s - \frac{in}{4s^2} \D \)D\(\frac{n}{2}\)\m{K}\f_n &= D\(\frac{n}{2}\)\m{K}\pa_s\f_n + s^{-1}D\(\frac{n}{2}\)\m{K}_1 \f_n \\
	&+ s^{\s-2} n D\(\frac{n}{2}\)\m{K}_2 \f_n.
\end{align*}

Since $\m{K}_1$ and $\m{K}_2$ of the form \eqref{def:Kpsi}, 
the estimate \eqref{eq:useful} is valid also for these regularizing operators.
Then, we have
\begin{align}
\begin{aligned}
	\Sobn{I_3}{\nu}
	&\le C|n|^{-\d+\eta} \int_t^{\I} s^{\frac{\nu-\d}{2}} \norm{\pa_s \phi_n(s)}_{H^{\d} \cap H^{0, -\d}} ds \\
	&+ C|n|^{-\d+\eta} \int_t^{\I} s^{\frac{\nu-\d}{2}-1}  \norm{\phi_n(s)}_{H^{\d} \cap H^{0, -\d}} ds \\
	&+ C|n|^{-\d+1+\eta} \int_t^{\I} s^{\frac{\nu-\d}{2}+\s-2}  \norm{\phi_n(s)}_{H^{\d} \cap H^{0, -\d}} ds.
\end{aligned}
\label{main:1:3}
\end{align}
By \eqref{main:1:1}, \eqref{main:1:2}, \eqref{main:1:3}, Lemma \ref{lem1:2} and the estimates
\begin{align*}
	\norm{\phi_n}_{H^{0,-\delta}}
	&\le C\norm{\wha{u_+}}_{L^\I}^{\frac2d} \norm{u_+}_{\dot{H}^{-\delta}}, \\
	\norm{\partial_t \phi_n}_{H^{0,-\delta}}
	&\le C\frac{|g_1|}{t}\norm{\wha{u_+}}_{L^\I}^{\frac4d} \norm{u_+}_{\dot{H}^{-\delta}},
\end{align*}
we find
\begin{multline*}
	\Sobn{\int_t^{\I} E^{1-\frac{1}{n}}(s) U\(\frac{n}{4s} \) D\(\frac{n}{2}\) \m{K}\f_n(s) \frac{ds}{s}}{\nu} \\
	\le C t^{\frac{\nu-\d}{2}+\s-1} |n|^{1+\eta} \norm{\wha{u_+}}_{L^\I}^{\frac2d} 
	\norm{u_+}_{\dot{H}^{-\d} \cap H^{0, d}}\Jbr{ g_1\norm{\wha{u_+}}_{L^\I}^{\frac2d} \log t}^\d \\
	\quad + C t^{\frac{\nu-\d}{2}} |n|^{1+\eta}|g_1| \norm{\wha{u_+}}_{L^\I}^{\frac4{d}} \norm{u_+}_{\dot{H}^{-\delta} \cap H^{0, d}}\Jbr{ g_1 \norm{\wha{u_+}}_{L^\I}^{\frac2{d}} \log t}^\d 
\end{multline*}

Thus, summing up with respect to $n$, we reach to the estimate
\begin{multline}\label{main:1:4}
	\Sobn{\int_t^{\I}U(t-s) \m{P}_d(s) ds}{\nu} \\
	\le Ct^{\frac{\nu-\d}{2}+\s-1} \Jbr{ g_1 \norm{\wha{u_+}}_{L^\I}^{\frac2{d}} \log t}^\d \Jbr{g_1\Lebn{\wha{u_+}}{\I}^{\frac{2}{d}}}\\
	\times   \Lebn{\wha{u_+}}{\I}^{\frac{2}{d}} \norm{u_+}_{\dot{H}^{-\delta} \cap H^{0, d}} \sum_{n \ne 0,1} |n|^{1+\eta} |g_n| .
\end{multline}
In a similar way, one sees that \eqref{main:1:4} holds true for $\nu=0$.
Therefore, in light of \eqref{eq:target}, we obtain
\begin{multline}\label{main:1:5}
	\norm{\int_t^{\I}U(t-s) \m{P}_d(s) ds}_{X_d} \\
	\le Ct^{-\frac54-\frac{\d}{2}+\s} \Jbr{ g_1 \norm{\wha{u_+}}_{L^\I}^{\frac2{d}} \log t}^\d \Jbr{g_1\Lebn{\wha{u_+}}{\I}^{\frac{2}{d}}}\\
	\times   \Lebn{\wha{u_+}}{\I}^{\frac{2}{d}} \norm{u_+}_{H^{0, \d} \cap \dot{H}^{-\d}} \sum_{n \ne 0,1} |n|^{1+\eta} |g_n| .
\end{multline}
By \eqref{main:1:4} with $\nu=0$ and \eqref{main:1:5}, we finally obtain
\begin{multline}\label{main:1:6}
	\norm{\int_t^{\I} U(t-s) \m{P}_{d} (s) ds}_{L^\I(T, \I; L^2)}+ \norm{\int_t^{\I} U(t-s) \m{P}_{d} (s) ds}_{L^4(T, \I; X_d)} \\
	\le CT^{-\frac{\d}{2}+\s-1} \Jbr{g_1 \norm{\wha{u_+}}_{L^\I}^{\frac2{d}} \log T}^\d \Jbr{g_1\Lebn{\wha{u_+}}{\I}^{\frac{2}{d}}}\\
	\times   \Lebn{\wha{u_+}}{\I}^{\frac{2}{d}} \norm{u_+}_{H^{0, \d} \cap \dot{H}^{-\d}} \sum_{n \ne 0,1} |n|^{1+\eta} |g_n|
\end{multline}
since $-\frac{\d}{2}+\s-1=-\gamma<0$.
The result follows from \eqref{non:n4} and \eqref{main:1:6}.

To complete the proof we prove Lemma \ref{useful:1:1}.

\begin{proof}[Proof of Lemma \ref{useful:1:1}]
It suffice to estimate $\dot{H}^\nu$ norm instead of $H^\nu$ norm because smaller $\nu$ gives better estimate
and because the case $\nu=0$ is included.
Further, we only treat the case $m=1$. 
We set $B = (1+t|\xi|^2)^{-\frac{1}{2}}$, which yields $|A(t)| \le CB^2$ for any $n\neq0,1$.
Since $\nu < 2-\d < 2-d/2$, we have $|\xi|^{\d}B^{2-\nu} \le Ct^{-\frac{\d}{2}}$
and $B^{2-\nu} \in L^2\cap L^\I(\R^d)$.
Set $\psi = U\(\frac{n}{4t} \) D\(\frac{n}{2}\) \m{K}\f_n(t)$.
By a standard argument, we have
\begin{equation}\label{eq:5.1_pf1}
	\norm{E^{1-\frac{1}{n}}(t)A(t)\psi}_{\dot{H}^{\nu}} \le C\norm{ |\pa|^{\nu} (A(t) \psi)}_{L^2} + Ct^{\frac{\nu}{2}}\norm{B^{2-\nu}\psi}_{L^2}.
\end{equation}

We first estimate the second term in \eqref{eq:5.1_pf1}. 
By the triangle inequality, 
\begin{align*}
	\Lebn{B^{2-\nu} U\(\frac{n}{4t} \) D\(\frac{n}{2}\) \m{K}\f_n(t)}{2} 
	\le{} &\Lebn{B^{2-\nu}\( U\(\frac{n}{4t} \) -1 \) D\(\frac{n}{2}\) \m{K}\f_n(t)}{2} \\
	&+ \Lebn{B^{2-\nu} D\(\frac{n}{2}\) \( \m{K}-1\) \f_n(t)}{2} \\
	&+ \Lebn{B^{2-\nu} D\(\frac{n}{2}\) \f_n(t)}{2}\\
	=:{} & \mathrm{I} + \mathrm{II}+ \mathrm{III}
\end{align*}
For any $p_1\in (2,\I]$, one sees from Sobolev embedding and Lemma \ref{mol:1.1} (i) that
\[
\begin{aligned}
	\Lebn{\mathrm{I}}{2}
	&\le C\Lebn{B^{2-\nu}}{p_1} \Lebn{|\nabla|^{\frac{d}{p_1}} \left| \frac{n|\n|^2}{t}\right|^{\frac12(\d-\frac{d}{p_1})} D\(\frac{n}{2}\)\m{K}\f_n(t)}{2} \\
	&\le Ct^{-\frac{\d}{2}} |n|^{-\d + (\frac{\d}2-\frac{d}{2p_1})} \norm{\f_n(t)}_{\dot{H}^{\d}}. 
\end{aligned}
\]
By definition of $\eta$, we are able to choose $p_1$ so that
\begin{equation}\label{eq:worstpoint}
	\frac{\d}2-\frac{d}{2p_1} < \eta.
\end{equation}
By Lemma \ref{mol:1.1} (ii), we estimate
\begin{align*}
	\Lebn{\mathrm{II}}{2} 
	&\le C\Lebn{B^{2-\nu}}{p_2} \Lebn{|\n|^{\frac{d}{p_2}}D\(\frac{n}{2}\)(\m{K}-1)\f_n(t)}{2} \\
	&\le Ct^{-\frac{d}{2p_2}} |n|^{-\frac{d}{p_2}} \Lebn{|\n|^{\frac{d}{p_2}}\(\m{K}-1\)\f_n(t)}{2} \\
	&\le Ct^{-\frac12(\frac{d}{p_2}- \t_2)} |n|^{-\frac{d}{p_2}- \t_2} \norm{\f_n(t)}_{\dot{H}^{\frac{d}{p_2}+\t_2}}
\end{align*}
for any $p_2\in (2,\I]$ and $\theta_2\in [0,1]$, where we have used the relation $\sigma\ge1$.
Taking $p_2$ and $\theta_2$ so that $\t_2 = \d-\frac{d}{p_2} \le1$, we obtain desired estimate for $\mathrm{II}$.
We can choose such $p_2$ and $\t_2$ because $\nu < 2-\delta$.
Next, we have
\begin{align*}
	&\Lebn{\mathrm{III}}{2} \le Ct^{-\frac{\d}{2}} \Lebn{|\xi|^{-\d}D\(\frac{n}{2}\) \f_n(t)}{2} \le Ct^{-\frac{\d}{2}} |n|^{-\d} \Lebn{|\xi|^{-\d}\f_n(t)}{2}.
\end{align*}
These estimates yield
\begin{equation}
t^{\frac{\nu}2}\Lebn{B^{2m-\nu} \psi}{2}
\le Ct^{\frac{\nu-\d}{2}}|n|^{-\d+\eta} \( \norm{\f_n(t)}_{H^{\d}} + \Lebn{|\xi|^{-\d}\f_n(t)}{2} \).
\label{est25:1}
\end{equation}

Let us move on to the estimate of the first term in \eqref{eq:5.1_pf1}. By interpolation inequality, 
\[
	\norm{ |\pa|^{\nu} (A(t) \psi)}_{L^2} \le \Lebn{A(t)\psi}{2}^{1-\nu}\Lebn{\n\(A(t)\psi\)}{2}^{\nu}.
\]
From $|\n A(t)| \le Ct^{1/2}B^2$ and the Leibniz rule, we have 
\[
	\Lebn{\n (A(t)\psi)}{2} \le Ct^{1/2}\Lebn{B^2\psi}{2} + \Lebn{B^2 \n \psi}{2}.
\] 
These implies that
\begin{align}
	\norm{ |\pa|^{\nu} (A(t) \psi)}_{L^2} \le Ct^{\frac{\nu}2}\Lebn{B^2\psi}{2} + C \Lebn{B^2\psi}{2}^{1-\nu} \Lebn{B^2 \n \psi}{2}^{\nu}. \label{com:2:1}
\end{align}
The estimate of $\Lebn{B^2\psi}{2}$ is the same as in \eqref{est25:1}.
To complete the proof, it then suffices to show that
\begin{align}
	\Lebn{B^2 \n \psi}{2} \le Ct^{\frac{1-\b}{2}}|n|^{-\b} \norm{\f_n(t)}_{\dot{H}^{\b}}, \label{com:2:2}
\end{align}
where $\b=1$ if $d=1$ and $\b=\d$ if $d=2$.

Let us show \eqref{com:2:2}. We estimate as
\begin{align*}
	\Lebn{B^{2}\n U\(\frac{n}{4t} \) D\(\frac{n}{2}\) \m{K}\f_n(t)}{2} 
	\le{} &\Lebn{B^{2}\n \( U\(\frac{n}{4t} \) -1 \) D\(\frac{n}{2}\) \m{K}\f_n(t)}{2} \\
	&+ \Lebn{B^{2}\n D\(\frac{n}{2}\) \( \m{K}-1\) \f_n(t)}{2} \\
	&+ \Lebn{B^{2}\n D\(\frac{n}{2}\) \f_n(t)}{2}\\
	=:{} & \mathrm{IV} + \mathrm{V}+ \mathrm{VI}.
\end{align*}
For any $p_{3} \in (4,\I]$, one sees from Sobolev embedding and Lemma \ref{mol:1.1} (i) that
\[
\begin{aligned}
	\Lebn{\mathrm{IV}}{2}
	&\le C\Lebn{B^{2}}{p_3} \Lebn{|\nabla|^{1+\frac{d}{p_3}} D\(\frac{n}{2}\)\m{K}\f_n(t)}{2} \\
	&\le Ct^{\frac{1-\b}{2}} |n|^{-\b} \norm{\f_n(t)}_{\dot{H}^{\b}}. 
\end{aligned}
\]
Here, we take $p_3$ so that $1+\frac{d}{p_3} = \b$.
By Lemma \ref{mol:1.1} (ii), we estimate
\begin{align*}
	\Lebn{\mathrm{V}}{2} 
	&\le C\Lebn{B^{2}}{p_4} \Lebn{|\n|^{1+\frac{d}{p_2}}D\(\frac{n}{2}\)(\m{K}-1)\f_n(t)}{2} \\
	&\le Ct^{-\frac{d}{2p_2}} |n|^{-1-\frac{d}{p_4}} \Lebn{|\n|^{1+\frac{d}{p_4}}\(\m{K}-1\)\f_n(t)}{2} \\
	&\le Ct^{\frac{1-\b}{2}} |n|^{-\b} \norm{\f_n(t)}_{\dot{H}^{\b}},
\end{align*}
where $\b = 1+ \frac{1}{p_4}$. Finally, from the Hardy inequality, we have
\begin{align*}
	\Lebn{\mathrm{VI}}{2} &\le Ct^{\frac{1-\b}{2}} \Lebn{|\xi|^{1-\d} |\n| D\(\frac{n}{2}\) \f_n(t)}{2} \\
	&\le Ct^{\frac{1-\b}{2}} |n|^{-\b} \Lebn{|\n|^{\b}\f_n(t)}{2}.
\end{align*}
By these estimates, we obtain \eqref{com:2:2}, which completes the proof of \eqref{eq:useful}.
\end{proof}

\section{Proof of main result}\label{sec:main_pf}

We are now in a position to prove our main result.

\begin{proof}[Proof of Theorem \ref{thm:main}]
Let $\eta>0$ and $\d\in(d/2,(d+1)/2)$ be as in the assumption. Then, we have the relation $\eta>\frac12(\d-d/2)$.
Take $b \in (d/4,\gamma)$.
By Lemma \ref{non1:g2}, Lemma \ref{non1:g1}, and Proposition \ref{non:im1}, we have
\begin{align}
\begin{aligned} 
&\norm{\P(v)}_{X_{d, T, b}} \\
&\le C_1\norm{g}_{\mathrm{Lip}} R \(R^{\frac{2}{d}}T^{\frac{1}{2}-\frac{2}{d}b}+\e^{\frac2d} \) \\
&+C_2(1+|g_1|)T^{b-\frac{\d}{2}}\Jbr{ g_1\e^{\frac{2}{d}} \log T }^\delta \Jbr{\e^{\frac{2}{d}}} \norm{u_+}_{H^{0,d}} \\
&+C_3T^{b-\gamma}\Jbr{g_1\e^{\frac{2}{d}}\log T}^{\d}\Jbr{g_1\e^{\frac{2}{d}}}  \eps^{\frac2d}\norm{u_+}_{\dot{H}^{-\d}\cap H^{0,d}}\sum_{n \ne 0,1}|n|^{1+\eta}|g_n| , 
\end{aligned} \label{est1:t1} 
\end{align}
for any $v \in X_{d, T, b, R}$, $R>0$, $T>T_0$ and $\e>0$.

We next see that 
\begin{equation}
d(\P(u), \P(v)) \le C_4 \norm{g}_{\mathrm{Lip}}\( R^\frac{2}{d} T^{\frac12 - \frac{2}{d}b} + \e^\frac{2}{d} \) d(u, v). \label{est1:t2}
\end{equation}
Indeed, by the integral equation of (NLS), we see that
\begin{align*}
\P(u) -\P(v) = i \int_t^{\I} U(t-s) \( F (u) - F (v)  \)(s) ds.
\end{align*}
We then find 
\begin{align*}
|F(u) - F(v)|
&\le C \norm{g}_{\mathrm{Lip}} \( |u|^{\frac{2}{d}} + |v|^{\frac{2}{d}} \)|u-v| \\
&\le C\norm{g}_{\mathrm{Lip}} \( |u-u_p|^{\frac{2}{d}} + |v-u_p|^{\frac{2}{d}} \)|u-v| \\
&\quad + \norm{g}_{\mathrm{Lip}} |u_p|^{\frac{2}{d}}|u-v|.
\end{align*}
The rest of the proof is similar to that of Lemma \ref{non1:g2}.

Choose $\e=\e(\norm{g}_{\mathrm{Lip}})$ so small that
\[
	C_1 \norm{g}_{\mathrm{Lip}} \eps^{\frac2d} \le \frac12,\quad
	C_4 \norm{g}_{\mathrm{Lip}} \eps^{\frac2d} \le \frac14.
\]
Set $R=1$. Then, for sufficiently large $T$, we obtain
\[
	\norm{\P(v)}_{X_{d, T, b}} <1=R
\]
and
\[
	d(\P(u), \P(v)) \le \frac12 d(u,v),
\]
which shows $\P: X_{d, T, b, 1} \rightarrow X_{d, T, b, 1}$ is a contraction mapping. 
Then, we obtain a unique solution $v(t)\in X_{d,t,b,1}$.
\end{proof}

\appendix
\section{Lipschitz continuity of $g_F$}\label{sec:A}

In this appendix we show the following.
\begin{lemma}
Let $F(u)$ satisfy \eqref{eq:cond1}. Let $g(\theta)$ be a corresponding periodic function given
by \eqref{eq:id1} and \eqref{eq:id2}. Then, the following two statements are equivalent:
\begin{enumerate}
\item $g(\theta)$ is Lipschitz continuous.
\item There exists $C>0$ such that 
\begin{equation}\label{eq:lip}
	|F(u)-F(v)| \le C(|u|^{2/d} + |v|^{2/d})|u-v|
\end{equation}
 for all $u,v \in \C$.
\end{enumerate}
Moreover, the constant $C$ depends only on the Lipschitz constant of $g$.
\end{lemma}
\begin{proof}
By \eqref{eq:id2}, it is easy to see that Lipschitz continuity of $g$ is equivalent to
existence of a constant $C$ such that
\begin{equation}\label{eq:lip1}
	|F(u)-F(v)| \le C |u-v|
\end{equation}
for all $u,v \in \C$ with $|u|=|v|=1$. Hence, (2)$\Rightarrow$(1) is obvious.

We will show that \eqref{eq:lip1} implies \eqref{eq:lip}. We may suppose that $u\neq0$ and $v\neq0$.
Otherwise \eqref{eq:lip} is immediate from \eqref{eq:cond1}. 
We have
\[
	|F(u)-F(v)| \le \abs{F(u) - F\(\frac{|u|}{|v|}v\)} + \abs{ F\(\frac{|u|}{|v|}v\) - F(v)}.
\]
By \eqref{eq:cond1} and \eqref{eq:lip1}, we have
\begin{align*}
	\abs{F(u) - F\(\frac{|u|}{|v|}v\)}
	&{}= |u|^{1+\frac2d}\abs{F\(\frac{u}{|u|}\) - F\(\frac{v}{|v|}\)} \\
	&{}\le C |u|^{1+\frac2d} \abs{ \frac{u}{|u|} - \frac{v}{|v|}} \\
	&{}\le C |u|^{1+\frac2d} \frac{||v|u-v|v||+||v|v-|u|v|}{|u||v|} \\
	&{}\le C |u|^{\frac2d} |u-v|.
\end{align*}
Again by \eqref{eq:cond1},
\begin{align*}
	\abs{ F\(\frac{|u|}{|v|}v\) - F(v)}
	&{}= \abs{F\(\frac{v}{|v|}\)} \abs{|u|^{1+\frac2d}-|v|^{1+\frac2d}} \\
	&{}\le C (|u|^{\frac2d} +|v|^{\frac2d}) |u-v|.
\end{align*}
Thus, we obtain \eqref{eq:lip}.
\end{proof}


\begin{bibdiv}
\begin{biblist}

\bib{B}{article}{
      author={Barab, Jacqueline~E.},
       title={Nonexistence of asymptotically free solutions for a nonlinear
  {S}chr\"odinger equation},
        date={1984},
        ISSN={0022-2488},
     journal={J. Math. Phys.},
      volume={25},
      number={11},
       pages={3270\ndash 3273},
         url={http://dx.doi.org/10.1063/1.526074},
      review={\MR{761850}},
}

\bib{GO}{article}{
      author={Ginibre, J.},
      author={Ozawa, T.},
       title={Long range scattering for nonlinear {S}chr\"odinger and {H}artree
  equations in space dimension {$n\geq 2$}},
        date={1993},
        ISSN={0010-3616},
     journal={Comm. Math. Phys.},
      volume={151},
      number={3},
       pages={619\ndash 645},
         url={http://projecteuclid.org/euclid.cmp/1104252243},
      review={\MR{1207269}},
}

\bib{HN02}{article}{
      author={Hayashi, Nakao},
      author={Naumkin, Pavel~I.},
       title={Large time behavior for the cubic nonlinear {S}chr\"odinger
  equation},
        date={2002},
        ISSN={0008-414X},
     journal={Canad. J. Math.},
      volume={54},
      number={5},
       pages={1065\ndash 1085},
         url={http://dx.doi.org/10.4153/CJM-2002-039-3},
      review={\MR{1924713}},
}

\bib{HN04}{article}{
      author={Hayashi, Nakao},
      author={Naumkin, Pavel~I.},
       title={On the asymptotics for cubic nonlinear {S}chr\"odinger
  equations},
        date={2004},
        ISSN={0278-1077},
     journal={Complex Var. Theory Appl.},
      volume={49},
      number={5},
       pages={339\ndash 373},
         url={http://dx.doi.org/10.1080/02781070410001710353},
      review={\MR{2073463}},
}

\bib{HN08}{article}{
      author={Hayashi, Nakao},
      author={Naumkin, Pavel~I.},
       title={Nongauge invariant cubic nonlinear {S}chr\"odinger equations},
        date={2008},
        ISSN={1941-3963},
     journal={Pac. J. Appl. Math.},
      volume={1},
      number={1},
       pages={1\ndash 16},
      review={\MR{2467127}},
}

\bib{HN11}{article}{
      author={Hayashi, Nakao},
      author={Naumkin, Pavel~I.},
       title={Global existence for the cubic nonlinear {S}chr\"odinger equation
  in lower order {S}obolev spaces},
        date={2011},
        ISSN={0893-4983},
     journal={Differential Integral Equations},
      volume={24},
      number={9-10},
       pages={801\ndash 828},
      review={\MR{2850366}},
}

\bib{HN15}{article}{
      author={Hayashi, Nakao},
      author={Naumkin, Pavel~I.},
       title={Logarithmic time decay for the cubic nonlinear {S}chr\"odinger
  equations},
        date={2015},
        ISSN={1073-7928},
     journal={Int. Math. Res. Not. IMRN},
      number={14},
       pages={5604\ndash 5643},
         url={http://dx.doi.org/10.1093/imrn/rnu102},
      review={\MR{3384451}},
}

\bib{HNST}{article}{
      author={Hayashi, Nakao},
      author={Naumkin, Pavel~I.},
      author={Shimomura, Akihiro},
      author={Tonegawa, Satoshi},
       title={Modified wave operators for nonlinear {S}chr\"odinger equations
  in one and two dimensions},
        date={2004},
        ISSN={1072-6691},
     journal={Electron. J. Differential Equations},
       pages={No. 62, 16 pp. (electronic)},
      review={\MR{2047418}},
}

\bib{HWN}{article}{
      author={Hayashi, Nakao},
      author={Wang, Huimei},
      author={Naumkin, Pavel~I.},
       title={Modified wave operators for nonlinear {S}chr\"odinger equations
  in lower order {S}obolev spaces},
        date={2011},
        ISSN={0219-8916},
     journal={J. Hyperbolic Differ. Equ.},
      volume={8},
      number={4},
       pages={759\ndash 775},
         url={http://dx.doi.org/10.1142/S0219891611002561},
      review={\MR{2864547}},
}

\bib{MM}{article}{
      author={Masaki, Satoshi},
      author={Miyazaki, Hayato},
       title={Global behavior of solutions to generalized {G}ross-{P}itaevskii
  equation},
        date={2016},
     journal={preprint},
     eprint={arXiv:1612.02738},
}

\bib{MS2}{article}{
      author={Masaki, Satoshi},
      author={Segata, Jun-ichi},
       title={Existence of a minimal non-scattering solution to the
  mass-subcritical generalized {K}orteweg-de {V}ries equation},
        date={2016},
     journal={preprint},
      eprint={arXiv:1602.05331},
}

\bib{MTT}{article}{
      author={Moriyama, Kazunori},
      author={Tonegawa, Satoshi},
      author={Tsutsumi, Yoshio},
       title={Wave operators for the nonlinear {S}chr\"odinger equation with a
  nonlinearity of low degree in one or two space dimensions},
        date={2003},
        ISSN={0219-1997},
     journal={Commun. Contemp. Math.},
      volume={5},
      number={6},
       pages={983\ndash 996},
         url={http://dx.doi.org/10.1142/S021919970300121X},
      review={\MR{2030566}},
}

\bib{N}{article}{
      author={Naumkin, P.~I.},
       title={The dissipative property of a cubic non-linear {S}chr\"odinger
  equation},
        date={2015},
        ISSN={0373-2436},
     journal={Izv. Ross. Akad. Nauk Ser. Mat.},
      volume={79},
      number={2},
       pages={137\ndash 166},
         url={http://dx.doi.org/10.4213/im8179},
      review={\MR{3352593}},
}

\bib{NS}{article}{
      author={Naumkin, Pavel~I.},
      author={S{\'a}nchez-Su{\'a}rez, Isahi},
       title={On the critical nongauge invariant nonlinear {S}chr\"odinger
  equation},
        date={2011},
        ISSN={1078-0947},
     journal={Discrete Contin. Dyn. Syst.},
      volume={30},
      number={3},
       pages={807\ndash 834},
         url={http://dx.doi.org/10.3934/dcds.2011.30.807},
      review={\MR{2784622}},
}

\bib{Oz}{article}{
      author={Ozawa, Tohru},
       title={Long range scattering for nonlinear {S}chr\"odinger equations in
  one space dimension},
        date={1991},
        ISSN={0010-3616},
     journal={Comm. Math. Phys.},
      volume={139},
      number={3},
       pages={479\ndash 493},
         url={http://projecteuclid.org/euclid.cmp/1104203467},
      review={\MR{1121130}},
}

\bib{MR3543568}{article}{
      author={Sagawa, Yuji},
      author={Sunagawa, Hideaki},
       title={The lifespan of small solutions to cubic derivative nonlinear
  {S}chr\"odinger equations in one space dimension},
        date={2016},
        ISSN={1078-0947},
     journal={Discrete Contin. Dyn. Syst.},
      volume={36},
      number={10},
       pages={5743\ndash 5761},
         url={http://dx.doi.org/10.3934/dcds.2016052},
      review={\MR{3543568}},
}

\bib{ShT}{article}{
      author={Shimomura, Akihiro},
      author={Tonegawa, Satoshi},
       title={Long-range scattering for nonlinear {S}chr\"odinger equations in
  one and two space dimensions},
        date={2004},
        ISSN={0893-4983},
     journal={Differential Integral Equations},
      volume={17},
      number={1-2},
       pages={127\ndash 150},
      review={\MR{2035499}},
}

\bib{St}{incollection}{
      author={Strauss, Walter},
       title={Nonlinear scattering theory},
        date={1974},
   booktitle={Scattering theory in mathematical physics},
      editor={Lavita, J.~A.},
      editor={Marchand, J.-P.},
   publisher={Reidel, Dordrecht, Holland},
       pages={53\ndash 78},
}

\bib{Sun}{article}{
      author={Sunagawa, Hideaki},
       title={Large time behavior of solutions to the {K}lein-{G}ordon equation
  with nonlinear dissipative terms},
        date={2006},
        ISSN={0025-5645},
     journal={J. Math. Soc. Japan},
      volume={58},
      number={2},
       pages={379\ndash 400},
         url={http://projecteuclid.org/euclid.jmsj/1149166781},
      review={\MR{2228565}},
}

\bib{TY}{article}{
      author={Tsutsumi, Yoshio},
      author={Yajima, Kenji},
       title={The asymptotic behavior of nonlinear {S}chr\"odinger equations},
        date={1984},
        ISSN={0273-0979},
     journal={Bull. Amer. Math. Soc. (N.S.)},
      volume={11},
      number={1},
       pages={186\ndash 188},
         url={http://dx.doi.org/10.1090/S0273-0979-1984-15263-7},
      review={\MR{741737}},
}

\end{biblist}
\end{bibdiv}


\end{document}